\documentclass{amsart}

\newtheorem{thm}{Theorem}[section]
\newtheorem{lemma}[thm]{Lemma}

\newtheorem{prop}[thm]{Proposition}

\begin{document}


\title[Chern subrings]{Chern subrings}

\author{Masaki Kameko} 
\address{Faculty of Contemporary Society, Toyama University of International Studies,
 Toyama, Toyama, Japan}
 \email{kameko@tuins.ac.jp}

\author{Nobuaki Yagita}
 \address{Faculty of Education, Ibaraki University,
Mito, Ibaraki, Japan}
\email{yagita@mx.ibaraki.ac.jp}




\begin{abstract}
Let $p$ be an odd prime.
We show that  for  a simply-connected semisimple complex linear algebraic  group, if its integral homology has $p$-torsion,  the Chern classes do not generate the Chow ring of  its classifying space.
\end{abstract}

\maketitle


\section{Introduction}

Let $p$ be an odd prime.
Let $h^{*}(-)$ be   one of  
the mod $p$ cohomology $H\mathbb{Z}/p$, the cohomology  $H\mathbb{Z}_{(p)}$ with coefficient $\mathbb{Z}_{(p)}$ and the Brown-Peterson cohomology $BP$ with $BP_*=\mathbb{Z}_{(p)}[v_1, v_2, \dots ]$. 
Let $G$ be a   compact connected Lie group and $G(\mathbb{C})$ it complexification, that is, $G(\mathbb{C})$ is a complex linear  algebraic group which is homotopy equivalent to the compact  connected Lie group $G$.
Considering a finite dimensional complex representation $\rho:G \to GL_{m}(\mathbb{C})$, we have  Chern classes $c_i(\rho)$ in the cohomology $h^*(BG)$ of classifying space and the Chern subring $Ch_{h}(G)\subset h^{*}(BG)$ , a subalgebra over $h_*$ generated by Chern classes, where $\rho$ ranges over all finite dimensional representations.
If $G$ is one of classical groups $SU(n)$, $\mathrm{Spin}(n)$ and $Sp(n)$,  the cohomology $h^{*}(BG)$ is generated by Chern classes and $h^{*}(BG)=Ch_{h}(G)$ for arbitrary odd prime $p$.

The case of  the Brown-Peterson cohomology  is particularly interesting in conjunction with the study of Chow rings of classifying spaces of complex linear algebraic groups defined by Totaro. In [To], Totaro considered  the classifying space of the linear algebraic group $G(\mathbb{C})$ as a limit of  algebraic varieties, defined the Chow ring for it and showed that the cycle map factors through the Brown-Peterson cohomology, 
\[
CH^{*}(BG(\mathbb{C}))_{(p)} \to BP^{*}(BG)\otimes_{BP_*} \mathbb{Z}_{(p)} \to H^{even}(BG;\mathbb{Z}_{(p)}),
\]
where $H^{even}(BG;\mathbb{Z}_{(p)})$ is the direct sum of $H^{2i}(BG;\mathbb{Z}_{(p)})$ $(i \geq 0)$.
He also conjectured that the left homomorphism $CH^{*}(BG(\mathbb{C}))_{(p)} \to BP^{*}(BG)\otimes_{BP_*} \mathbb{Z}_{(p)}$ is an isomorphism. We may consider a Chern subring  for the Chow ring $CH^*(BG(\mathbb{C}))$ as in the case of the above $Ch_{h}(G)$.

In \cite{Ka-Ya} and \cite{Vi}, the Chow ring $CH^*(BPGL_p(\mathbb{C}))_{(p)}$ of the complex linear algebraic group $PGL_p(\mathbb{C})$, which is the complexification of the projective unitary group $PU(p)$,  and related cohomology theories were computed and  it was shown that
 $$CH^*(BPGL_p(\mathbb{C}))_{(p)}=BP^{*}(BPU(p)) \otimes_{BP_*} \mathbb{Z}_{(p)}=H^{even}(BPU(p);\mathbb{Z}_{(p)})$$ through the cycle map above.
 In \cite[Proposition~5.7]{Ka-Ya}, we showed similar results  for $(p,G)=(3, F_4)$, $(5, E_8)$.
For $p=3$, the computation of the Brown-Peterson cohomology was done by Kono and Yagita in \cite{Ko-Ya} and Kono and  Yagita showed that  $x_8^{a}$ is not in the Chern subring unless $a$ is divisible by $2$. In \cite{Ta}, Targa  showed that $x_{2p+2}^{a}$ in $CH^{*}(BPGL_p(\mathbb{C}))_{(p)}$, where $a\leq p-2$,  is not in the Chern subring for arbitrary odd prime $p$.

In this paper, we prove the following and generalize the above computation of Kono, Yagita and Targa.
Let $Q_i$ be the Milnor operations of degree $2p^i-1$ which acts on the mod $p$ cohomology of a space.


\begin{thm} \label{Theorem:main} 
For $(p,G)=(p, PU(p))$, let $x=Q_0Q_1 x_2$ where $x_2$ is the generator of $H^{2}(BG;\mathbb{Z}/p)=\mathbb{Z}/p$. For $(p,G)=(3, F_4)$, $(3, E_6)$, $(3, E_7)$, $(3, E_8)$, $(5, E_8)$, let $x=Q_1Q_2 x_4$ where $x_4$ is the generator of $H^{4}(BG;\mathbb{Z}/p)=\mathbb{Z}/p$. Then, $x^a$ is not in the Chern subring $Ch_{H\mathbb{Z}/p}(G)$ unless $a$ is  divisible by $p-1$.
\end{thm}

This theorem implies that if $x$ comes from the Chow ring through the cycle map, then the Chow ring is not generated by Chern classes. Recall that motivic cohomology $H^{*, *'}(BG(\mathbb{C}), \mathbb{Z}/p)$ contains $CH^{*}(BG(\mathbb{C}))/p$ as $$CH^{*}(BG(\mathbb{C}))/p=H^{2*, *}(BG(\mathbb{C}), \mathbb{Z}/p).$$
Moreover, the motivic cohomology has the action of Milnor operations $Q_i$ where the degree of $Q_i$ is 
$(2p^i-1, p^i-1)$.
If there exists an element $x_{4,3}$ in 
$H^{4,3}(BG(\mathbb{C}), \mathbb{Z}/p)$  corresponding to 
$x_4$ in $H^{4}(BG;\mathbb{Z}/p)$, 
then $x=Q_1Q_2(x_{4,3})$ is in the Chow ring 
$$CH^{p^2+p+1}(BG(\mathbb{C}))/p=H^{2p^2+2p+2,p^2+p+1}(BG(\mathbb{C}),\mathbb{Z}/p)$$ 
and through the cycle map it maps to $x$ in Theorem~\ref{Theorem:main}.
In \cite{Ya}, Lemma 9.6, Yagita proved that if $px_4\in H^{4}(BG;\mathbb{Z}_{(p)})$ is a Chern class of some representation, then the element $x_{4,3}$ above exists. In \cite{Sc-Ya}, Schuster and Yagita  showed that for $(p, G)=(3, F_4)$, 
$3 x_{4}$ is the Chern class of the complexification of the irreducible representation of $F_4$.
In this paper, by computing the Chern class of the adjoint representation of $E_8$, we prove the following proposition.


\begin{prop}\label{Proposition:chern}
For $(p, G)=(3, F_4)$, $(3, E_6)$, $(3, E_7)$, $(3, E_8)$ and $(5, E_8)$,
there exists a complex representation $\alpha$ of $G$ and $\gamma\in \mathbb{Z}_{(p)}^{\times}$ such that the element $\gamma px_4 \in H^{4}(BG;\mathbb{Z}_{(p)})$ is a Chern class $c_2(\alpha)$ \end{prop}

Thus, we have the following result on Chern subrings of Chow rings.


\begin{thm} \label{Theorem:chow} For $(p, G)=(p, PU(p))$, $(3, F_4)$,  $(3, E_6)$, $(3, E_7)$, $(3,E_8)$ and $(5,E_8)$, 
the Chow ring $CH^{*}(BG(\mathbb{C}))_{(p)}$ is not generated by Chern classes.
\end{thm}

In \S2, we consider Chern classes of elementary abelian $p$-groups. In \S3, we prove Theorem~\ref{Theorem:main}.
In \S4, we prove Proposition~\ref{Proposition:chern}.
We thank Fran\c{c}ois-Xavier Dehon for informing us of the work of Targa.


\section{Chern classes of elementary abelian $p$-groups}

In this section, we investigate the total Chern class of finite dimensional complex representation $\rho:A_n \to GL_m(\mathbb{C})$ of elementary abelian $p$-group $A_n$ of rank $n$.

Firstly, we recall the cohomology of $BA_n$. The mod $p$ cohomology of elementary abelian $p$-group is a polynomial tensor exterior algebra $$\mathbb{Z}/p[t_1, \dots, t_n] \otimes \Lambda( dt_1, \dots, dt_n).$$
The elements $dt_1, \dots, dt_n \in H^1(BA_n;\mathbb{Z}/p)$ correspond to the dual of the basis of $\pi_1(BA_n)=H_1(BA_n;\mathbb{Z}/p)$. The elements $t_1, \dots, t_n$ are obtained from $dt_1, \dots, dt_n$ by applying the Milnor operation $Q_0$. 
For the mod $p$ cohomology of a space, there exists an action of Milnor operations $Q_0, Q_1, Q_2, \dots$ and reduced power operations $\wp^0=1, \wp^1, \wp^2, \dots$. The action of Milnor operations on the mod $p$ cohomology of elementary abelian $p$-group is given by
\[
Q_i(dt_k)=t^{p^i}_{k},\quad
Q_i t_k =0, \quad Q_i(x\cdot y)=Q_i(x) \cdot y + (-1)^{\deg x} x \cdot Q_i(y).
\]
The action of reduced power operations is given by
\[
\wp^{i} dt_k=0,\quad 
\wp^{i} t_k=\begin{cases} t_k^p & (i=1) \\ 0 & (i\geq 2), \end{cases} \quad
\wp^{j}(x \cdot y)=\sum_{i=0}^{j} \wp^{i-j} x \cdot \wp^{j} y.
\]

Secondly, we recall the invariant theory of finite general linear groups and special linear groups. The action of Milnor operations commutes with the action of general linear group $GL_n(\mathbb{Z}/p)$ since the action of the general linear group on the mod $p$ cohomology comes from the one on the elementary abelian $p$-group $A_n$.
For the sake of notational simplicity,  we write $V_n$ for the subspace spanned by $t_1, \dots, t_n$, $$V_n=\mathbb{Z}/p\{ t_1, \dots, t_n\}.$$
We write $SM_n$, $M_n$ for M\`{u}i invariants
\[
H^{*}(BA_n;\mathbb{Z}/p)^{SL_n(\mathbb{Z}/p)}, \quad H^{*}(BA_n;\mathbb{Z}/p)^{GL_n(\mathbb{Z}/p)}, 
\]
respectively.
We also write 
 $SD_n$, $D_n$ for Dickson invariants
 \[
 \mathbb{Z}/p[t_1, \dots, t_n]^{SL_n(\mathbb{Z}/p)}, \quad \mathbb{Z}/p[t_1, \dots, t_n]^{GL_n(\mathbb{Z}/p)}, 
\]
respectively. 
Kameko and Mimura \cite{Ka-Mi} gave a simpler description for $SM_n$, $M_n$ using Milnor operations. For Dickson invariants and M\`{u}i invariants, we refer the reader to \cite{Ka-Mi} and its references.
Let us define $c_{n, i}$ for $n=1, \dots, n-1$ as follows:
Consider the polynomial 
\[
f_n(X)=\displaystyle \prod_{v \in V_n} (X+v)
\]
in $\mathbb{Z}/p[t_1, \dots, t_n][X]$. We define $(-1)^{n-i}c_{n, i}$ to be the coefficient of $X^{p^{n-i}}$ in $f_n(X)$.
We define $e_n$ by  $e_n=Q_0\cdots Q_{n-1} (dt_1\cdots dt_n)$. 
Then, we have the following.
For a ring $R$ and for a finite set $\{a_1, \dots, a_r\}$, we denote by $R\{ a_1, \dots, a_r\}$ a free $R$-module with the basis $\{a_1, \dots, a_r\}$.


\begin{prop}\label{Proposition:e} There hold the following:\\
{\rm (1)} $c_{n,0}=e_n^{p-1}$.\\
{\rm (2)} $f_n(X)=X^{p^n}-c_{n,n-1}X^{p^{n-1}}+\cdots +(-1)^n c_{n,0} X$, \\
{\rm (3)} $SD_n$ is a polynomial algebra $\mathbb{Z}/p[e_n, c_{n,n-1}, \dots, c_{n,1}]$.\\
{\rm (4)} $D_n$ is also a polynomial algebra $\mathbb{Z}/p[c_{n,n-1}, \dots, c_{n,1}, c_{n,0}]$.\\
{\rm (5)} $M_n$ is a free $D_n$-module $$D_n \{ 1, e_n^{p-2} dt_1\dots dt_n, e_{n}^{p-2} Q_{i_1}\dots Q_{i_r} (dt_1\cdots dt_n) \} \quad \text{and}$$
{\rm (6)} $SM_n$ is a free $SD_n$-module  $$SD_n \{ 1, dt_1\dots dt_n, Q_{i_1}\dots Q_{i_r} (dt_1\cdots dt_n) \},$$where $0\leq i_1<\cdots<i_r\leq n-1$, $1\leq r \leq n-1$.
\end{prop}

Thirdly, we consider Chern classes. 
 It is well-known that any finite dimensional complex representation of an abelian group is a direct sum of $1$-dimensional complex representations. 
Therefore, the total Chern class $c(\rho)$ is a product of $c(\lambda)$'s where $c(\lambda)=1+v$, $v \in V_n$. Thus, the Chern classes are in $\mathbb{Z}/p[t_1, \dots, t_n]$ instead of $H^{*}(BA_n;\mathbb{Z}/p)$.
Let us consider the total Chern class $c(reg)$ of the regular representation $reg:A_n \to GL_{p^n}(\mathbb{C})$.
It is clear that $GL_n(\mathbb{Z}/p)$ acts on $A_n$ and $c(reg) \in M_n$.

\begin{prop}  \label{Proposition:r} There holds
\[
c(reg)=\prod_{v \in V_n\backslash\{0\}} (1+v)=1-c_{n,n-1}+\dots+(-1)^{n}c_{n,0} \in D_n.
\]
\end{prop}

For a group $W$ acting $V_n\backslash\{0\}$, we say the action of $W$ is transitive on $V_n\backslash\{0\}$ if and only if 
for each $u, v$ in $V_n\backslash\{0\}$, there exists $w \in W$ such that $w u=v$. We investigate the total Chern class $c(\rho)$ when the image of  the induced homomorphism $B\rho^*:H^{*}(BGL_{m}(\mathbb{C});\mathbb{Z}/0) \to \mathbb{Z}/p[t_1, \dots, t_n]$ is invariant under certain group action.


 \begin{lemma} \label{Lemma:image}
 Let $\rho:A_n \to GL_m(\mathbb{C})$ be a complex representation of elementary abelian $p$-group $A_n$ of rank $n$. Suppose that a subgroup $W$ of $GL_n(\mathbb{Z}/p)$ acts on $A_n$ in the obvious manner. Suppose that 
 the total Chern class $c(\rho)$ is in $\mathbb{Z}/p[t_1, \dots, t_n]^{W}$ and suppose that the 
 action of $W$ on $V_n\backslash\{0\}$ is transitive.  Then, $c(\rho)=c(reg)^{a}$ for some $a\geq 0$. \end{lemma}
 
 \begin{proof}
 Suppose that 
 \[
 c(\rho)=\prod_{v \in V_n\backslash\{0\}} (1+v)^{\mu(v)}.
 \]
The non-negative integer $\mu(v)$ is the divisibility of $c(\rho)$ by $1+v$. In other words, 
$c(\rho)$ is divisible by $(1+v)^{\mu(v)}$ but not divisible by $(1+v)^{\mu(v)+1}$.
In order to prove the lemma, it suffices  to show that $\mu(v)$ is a constant function of $v \in V_n\backslash \{0\}$. 
 Suppose that $\mu(u) <\mu(v)$ for some $u, v \in V_n \backslash\{0\}$.
 Let $w \in W$ be an element such that $w v=u$ . Then, since $w$ acts trivially on $c(\rho)$, we have
 \[
c(\rho)=w  c(\rho) = \prod_{v' \in V_n \backslash\{0\}}  (w(1+v'))^{\mu(v')}=
\left(\prod_{v'  \in V_n\backslash\{0, v\}} (1+wv')^{\mu(v')}\right) (1+u)^{\mu(v)}.
\]
This implies that $\mu(u)\geq \mu(v)$. It is a contradiction. Hence, we have the desired result.
 \end{proof}
 
By Proposition~\ref{Proposition:r} and Lemma~\ref{Lemma:image}, we have the following result:

\begin{prop}
Let $G$ be a compact connected Lie group and let $A_n$ be an elementary abelian $p$-subgroup of $G$.
Suppose  that the Weyl group of $A_n$, that is the quotient of the normalizer of $A_n$ in $G$ by the centralizer of $A_n$ in $G$, acts transitively on $V_n\backslash\{0\}$.  Then, $B\eta^*(Ch_{H\mathbb{Z}/p}(G))\subset D_n$, where  $\eta:A_n \to G$ be the inclusion of $A_n$ into $G$. 
\end{prop}

We end this section by recalling the following fact:

\begin{prop}
The action of $SL_n(\mathbb{Z}/p)$ on $V_n\backslash\{0\}$ is transitive for $n\geq 2$.
\end{prop}

\begin{proof}
It is an easy exercise of linear algebra.  It suffices to show that for any $a=(a_1, a_2, \dots, a_n) \in V_n\backslash \{0\}$, there exists a matrix $g$ in $SL_{n}(\mathbb{Z}/p)$ such that
\[
g \left( \begin{array}{c} 1 \\ 0 \\ \vdots \\ 0 \end{array}\right)={}^t a, \]
where ${}^t a$ is the transpose of $a$.
If necessary, applying a permutation, without loss of generality, we may assume that $a_1 \not =0$. We choose the first column vector of $g$ to be ${}^t a$ and the first  row vector of $g$ to be $(a_1, 0, \dots, 0)$ and we choose the rest of the entries in the matrix $g$ so that the matrix obtained from $g$ by removing the first column and the first row is a diagonal matrix whose $(i,i)$-entry is $1$ for $i=1, \dots, n-2$ and $(n-1, n-1)$-entry is $a_{1}^{-1}$.
Then by computing the cofactor expansion along the first row, we see that the determinant of $g$ is $1$ and so  $g$ is in $SL_n(\mathbb{Z}/p)$. By definition, it is clear that $g$ satisfies the required equality.
\end{proof}

Thus, in order to prove Theorem~\ref{Theorem:main}, it suffices to show that there exists an elementary abelian $p$-subgroup $A_n$ whose Weyl group is $SL_n(\mathbb{Z}/p)$ and that $B\eta^*(x)\not\in D_n$. This is what we do in the next section.


\section{Chern subrings}

In this section, we prove Theorem~\ref{Theorem:main} by observing the cohomology of non-toral elementary abelian $p$-subgroup of $G$.
There exist non-toral elementary abelian $p$-subgroups in a compact connected Lie group if the integral homology of the  Lie group has $p$-torsion. These non-toral elementary abelian $p$-subgroups and their Weyl groups are known for $(p, G)=(p, PU(n)), (3, F_4), (3, E_6), (3, E_7), (3, E_8), (5, E_8)$. We refer the reader to Andersen et al. \cite{A-G-M-V} and its references.
 In this paper, we use the following results for $(p, G)=(p, PU(p)), (3, F_4)$ and $(5, E_8)$ only:

\begin{prop} There hold the following: \newline
{\rm (1)} For $(p, G)=(p, PU(p))$, there exists a non-toral elementary abelian $p$-subgroup $A_2$ of rank $2$ such that its Weyl group in $G$ is the special linear group $SL_2(\mathbb{Z}/p)$. \\
{\rm (2)} For $(p, G)=(3, F_4), (5, E_8)$, there exists a non-toral elementary $p$-subgroup $A_3$ of rank $3$ such that its Weyl group in $G$ is the special linear group $SL_3(\mathbb{Z}/p)$.
\end{prop}

Let $\eta:A_n \to G$ be the inclusion of non-toral elementary abelian $p$-subgroup in $G$. In  \cite{Ka-Ya}, we computed  the image of  the induced homomorphism 
\[
B\eta^*:H^{*}(BG;\mathbb{Z}/p) \to SM_n
\]
 for $(p,G)=(p,PU(p))$, $n=2$ and for $(p,G)= (3, F_4)$, $(3, E_6)$, $(3, E_7)$, $(5,E_8)$, $n=3$.
Since we wish to include the case $(p,G)=(3, E_8)$ in Theorem~\ref{Theorem:main}, instead of making use of the computation of the image of  $B\eta^*$, we use the following result, which is also used in the computation of the image of $B\eta^*$:

\begin{prop} \label{iso}
There hold the following:\\
{\rm (1)} The induced homomorphism
$$
H^{2}(BPU(p);\mathbb{Z}/p)\to SM_2^2=\mathbb{Z}/p\{ dt_1dt_2\}
$$
is an isomorphism.\\
{\rm (2)} For $(p,G)=(3, F_4)$ and $(5, E_8)$, the induced homomorphism
$$
H^{4}(BG;\mathbb{Z}/p) \to SM_3^4=\mathbb{Z}/p\{ Q_0 (dt_1dt_2dt_3)\}
$$
is an isomorphism.
\end{prop}

Now, we prove Theorem~\ref{Theorem:main}  for $(p,G)=(3, E_8)$.  As we mentioned at the end of the previous section, 
it suffices to show that  $B\eta^*(x) \not \in D_3$.
There is a sequence of inclusions 
\[
F_4 \to E_6 \to E_7 \to E_8
\]
and the induced homomorphisms 
\[
H^{4}(BF_4;\mathbb{Z}/p) \leftarrow 
H^{4}(BE_6;\mathbb{Z}/p) \leftarrow 
H^{4}(BE_7;\mathbb{Z}/p) \leftarrow 
H^{4}(BE_8;\mathbb{Z}/p) =\mathbb{Z}/p
\]
are isomorphisms.
Recall that we denote the generator of $H^{4}(BE_8;\mathbb{Z}/3)$ by $x_4$. We define $x\in H^{26}(BE_8;\mathbb{Z}/3)$ by $x=Q_1Q_2(x_4)$.
Since the induced homomorphism maps $x_4$ to $Q_0 (dt_1dt_2dt_3)$ by Proposition~\ref{iso}, 
it maps $x$ to $e_3=Q_0Q_1Q_2(dt_1dt_2dt_3)$ in $SD_3$.
It is clear that $e_3^{a}$ is not in $D_3$ unless $a$ is divisible by $p-1$. Thus, we have Theorem~\ref{Theorem:main}  for $(p,G)=(3, E_8)$. 
Theorem~\ref{Theorem:main}  for the other $(p,G)$'s can be proved in the same manner.


\section{Proof of Proposition~\ref{Proposition:chern}}

In this section, we prove Proposition~\ref{Proposition:chern} by computing the second Chern class of the adjoint representation of the exceptional Lie group $\alpha:E_8\to SO(248)$. 
Similar computation was done in \cite{Sc-Ya} for the irreducible representation $F_4 \to SO(26)$.

Since the induced homomorphism 
\[
H^4 (BF_4;\mathbb{Z}_{(3)}) \leftarrow H^4 (BE_6;\mathbb{Z}_{(3)}) \leftarrow H^4 (BE_7;\mathbb{Z}_{(3)}) \leftarrow H^4 (BE_8;\mathbb{Z}_{(3)})=\mathbb{Z}_{(3)}
\]
are isomorphisms, if $3 x_4$ in $H^{4}(BE_8;\mathbb{Z}_{(3)})$ is a Chern class, so is in $H^{4}(BG;\mathbb{Z}_{(3)})$
for $G=F_4, E_6, E_7$.  So, it suffices to show the proposition for $G=E_8$.

Let $\alpha:E_8 \to SO(248)$ be the adjoint representation of $E_8$. 
By the construction of the exceptional Lie group $E_8$ in \cite{Ad}, there exists a homomorphism $\beta:\mathrm{Spin}(16)\to E_8$ such that the induced representation  $\alpha\circ \beta $ is 
the direct sum of $\lambda^2_{16}:\mathrm{Spin}(16) \to SO(120)$ and $\Delta^{+}_{16}:\mathrm{Spin}(16)\to SO(128)$.  See \cite[Corollary 7.3]{Ad} and \cite[p. 143]{Mi-Ni}.
Let $T^8$ be the maximal torus of $\mathrm{Spin}(16)$. 
Let $T^1$ be the first factor of $T^8$ and 
$\eta:T^1\to \mathrm{Spin}(16)$ the inclusion of $T^1$ into 
$\mathrm{Spin}(16)$. 
Denote by $R(G)$ the complex representation ring of $G$.
The complexification of  $\lambda^2_{16}$ corresponds to the second elementary symmetric function of $z_1^2+z_1^{-2}, \dots, z_{8}^2+z_{8}^{-2}$ in $R(T^8)$ and the complexification of $\displaystyle \Delta_{16}^{+}$ corresponds to $\displaystyle \sum_{\varepsilon_1\cdots \varepsilon_8=1} z_1^{\varepsilon_1}\dots z_{8}^{\varepsilon_8}$ in $R(T^8)$, where  $\varepsilon_r =\pm 1$ for $r=1,\dots, 8$. 

So, the restriction of the complexification of $\lambda^2_{16}$ to $T^1$  corresponds to  $$2^{2}\left( \begin{array}{c} 7 \\ 2 \end{array}
\right)+2\left( \begin{array}{c} 7 \\ 1 \end{array}
\right) (z_1^{2}+z_{1}^{-2})=84+14(z_1^{2}+z_{1}^{-2}) \quad \text{in $R(T^1)$}.$$
The restriction of the complexification of $\Delta_{16}^{+}$ to $T^1$ corresponds to  $$2^6 (z_1+z_1^{-1})=64(z_1+z_1^{-1})\quad \text{in $R(T^1)$.}$$
Therefore, the total Chern class of the complexification of $\alpha \circ \beta \circ\eta$ is $$\{(1+2u)(1-2u)\}^{14}\{(1+u)(1-u)\}^{64}=1-120u^2+\cdots \in \mathbb{Z}[u]=H^*(BT^1;\mathbb{Z}), $$
where $u$ is the generator of $H^{2}(BT^1;\mathbb{Z})=\mathbb{Z}$.
Since $120=2^3 \cdot 3 \cdot 5$,  the Chern class $c_2(\alpha)$ represents $\gamma px_4$ for $p=3, 5$ in $H^{4}(BE_8;\mathbb{Z}_{(p)})$, where $\gamma$ is a unit in $\mathbb{Z}_{(p)}$ and $x_4$ is the generator of $H^{4}(BE_8;\mathbb{Z}_{(p)})=\mathbb{Z}_{(p)}$.
This completes the proof of Proposition~\ref{Proposition:chern}.



\end{document}